\documentclass[letterpaper, 10 pt, conference]{ieeeconf}  

\IEEEoverridecommandlockouts                              
\usepackage{times} 
\usepackage{amsmath,amsfonts} 
\usepackage{amssymb}  
\usepackage{tabularx}
\usepackage[normalem]{ulem}
\usepackage{dsfont}
\usepackage{graphicx}
\usepackage[dvipsnames,usenames]{xcolor}
\usepackage[colorlinks,%
linkcolor=BrickRed,%
filecolor=BrickRed,%
citecolor=RoyalPurple,%
]{hyperref}
\usepackage{algorithm}
\usepackage{algorithmic}
\usepackage{float}

\def\Re{\mathbb{R}}

\def\notes#1{\marginpar{\tiny #1}\typeout{Notes!
Notes!
Notes!
}}
\renewcommand{\notes}[1]{\typeout{notes!}}

\def\Re{\field{R}}

\def\Expect{{\mathbb E}}

\newtheorem{theorem}{Theorem}

\newtheorem{lemma}{Lemma}
\newtheorem{remark}{Remark}
\newtheorem{proposition}{Proposition}

\newtheorem{assumption}{Assumption}

\def\beq{\begin{eqnarray}} 
\def\bc{\begin{center}} 
\def\be{\begin{enumerate}}
\def\bi{\begin{itemize}} 
\def\bs{\begin{small}}
\def\bS{\begin{slide}}
\def\ec{\end{center}} 
\def\ee{\end{enumerate}}
\def\ei{\end{itemize}}
\def\es{\end{small}}
\def\eS{\end{slide}}
\def\eeq{\end{eqnarray}}

\newcommand{\newP}[1]{\noindent{\bf #1:}}

\def\Re{\mathbb{R}}

\renewcommand{\Re}{\mathbb{R}}

\newcounter{rmnum}

\newcounter{anum}

\title{\LARGE \bf
	Data-Driven Approximation of Stationary Nonlinear Filters \\with Optimal Transport Maps
}

\author{Mohammad Al-Jarrah$^\star$, Bamdad Hosseini$^\dagger$, Amirhossein Taghvaei$^\star$
\thanks{Mohammad Al-Jarrah and Amirhossein Taghvaei are supported by the National Science Foundation (NSF) award EPCN-2318977. Bamdad Hosseini is supported by the NSF award DMS-2208535}
    {\thanks{$^\star$Department of Aeronautics \& Astronautics, University of Washington, Seattle; {\tt\small mohd9485@uw.edu,amirtag@uw.edu}.}}
    {\thanks{$^\dagger$Department of Applied Mathematics, University of Washington, Seattle
        {\tt\small bamdadh@uw.edu}.}}
}

\begin{document}

      \maketitle
	 \thispagestyle{empty}
	 \pagestyle{empty}


	 \begin{abstract}
          The nonlinear filtering problem is concerned with finding the conditional probability distribution (posterior) of the state of a stochastic dynamical system, given a history of partial and noisy observations. This paper presents a data-driven nonlinear filtering algorithm for the case when the state and observation processes are stationary. The posterior is approximated as the push-forward of an optimal transport (OT) map from a given distribution, that is easy to sample from, to the posterior conditioned on a truncated observation window. The OT map is obtained as the solution to a stochastic optimization problem that is solved offline using recorded trajectory data from the state and observations. An error analysis of the algorithm is presented under the stationarity and filter stability assumptions, which decomposes the error into two parts related to the truncation window during training and the error due to the optimization procedure. The performance of the proposed method, referred to as optimal transport data-driven filter (OT-DDF), is evaluated for several numerical examples, highlighting its significant computational efficiency during the online stage while maintaining the flexibility and accuracy of OT methods in nonlinear filtering. 

	\end{abstract}

    \section{Introduction}
A  nonlinear filtering problem consists of two processes: (i) a hidden Markov process $\{X_1,X_2,\ldots\} $ that represents the state of a dynamical system; and (ii) an observed random process $\{Y_1,Y_2,\ldots\}$ that represents the noisy sensory measurements of the state. The job of a nonlinear filter is to numerically approximate the posterior distribution, i.e. the conditional probability distribution of the state $X_t$ given a history of noisy measurements $\{Y_t,Y_{t-1},\ldots, Y_1\}$, for $t=1,2,\ldots$. The exact posterior admits a recursive update law that facilitates the design of nonlinear filtering algorithms~\cite{cappe2009inference}. Denoting the posterior at time $t$ by $\pi_t$, the recursive update law may be expressed as
\begin{equation}\label{eq:filter-update}
    \pi_{t} = F_t(Y_t)(\pi_{t-1})
\end{equation}
where $F_t(Y_t)$ is a $Y_t$-dependent map on the space of probability distributions that consists of two operations: the propagation step that updates the posterior according to the dynamic model and the conditioning step that updates the posterior according to Bayes' rule; see Sec.~\ref{sec:filtering-problem} for details and a brief review. The initial distribution $\pi_0$ is the probability distribution of the initial condition $X_0$. 

Nonlinear filtering algorithms carry out different numerical approximations of the update step~\eqref{eq:filter-update}. Kalman filter (KF)~\cite{kalman1960new}, and its extensions~\cite{bar2004estimation,evensen2003ensemble,calvello2022ensemble}, rely on a Gaussian approximation of the joint distribution of the state and observation, thereby, simplifying~\eqref{eq:filter-update} to an update for a mean vector and a covariance matrix.  Due to the Gaussian approximation, the performance of KF type algorithms is limited to linear dynamical systems with additive Gaussian noise. Sequential importance re-sampling particle filters~~\cite{gordon1993novel,doucet09}
approximate the posterior with a weighted empirical distribution of a large number of particles. While they provide an asymptotically exact solution in the limit of infinitely many particles, they suffer from the weight-degeneracy issues in high dimensions~\cite{bickel2008sharp,bengtsson08,rebeschini2015can}. Coupling and controlled interacting particle system approaches~\cite{daum10,taoyang_TAC12,reich2013nonparametric,de2015stochastic,yang2016,marzouk2016introduction,spantini2022coupling,shi2022conditional,taghvaei2023survey}
avoid weight degeneracy by replacing the importance sampling step with a control law/coupling that updates the location of the particles with uniform weights. However, the main bottleneck of these types of algorithms becomes the online computation of the control law/coupling. 

This paper is built upon our recent work that proposes an optimal transport (OT)  variational formulation of the Bayes' law to construct nonlinear filtering algorithms~\cite{taghvaei2022optimal,al2023optimal,grange2023computational,al2023highdim}.
In this formulation, the update step~\eqref{eq:filter-update} is replaced with a push-forward of a map $T_t$ following 
\begin{equation}
    \pi_{t} = T_t(\cdot,Y_{t})_{\#} \pi_{t-1},
\end{equation}
and the map $T_t$ is obtained by solving a stochastic optimization problem that aims at finding the OT map from $\pi_{t-1}$ to $\pi_{t}$ (see Sec.~\ref{sec:OT-formulation} for details).
This approach, which is referred to as the OT particle filter (OTPF), has two main appealing features: (i) it only requires samples from the joint distribution of the state $X_t$ and the observations $Y_t$, without the need for the analytical model of the observation likelihood and dynamics, i.e. given a sample $X_{t-1}$, we need an oracle or a simulator that samples $X_{t}$ and $Y_{t}$; (ii)  it allows for the utilization of neural networks to enhance the representation power of the transport maps $T_t$ to model complex and multi-modal probability distributions. Due to these features, the OTPF is numerically favorable for problem settings that involve multi-modal and high-dimensional posterior distributions~\cite{al2023highdim}. However, the better performance comes with the cost of solving a stochastic optimization problem online at each time $t$. 

In this paper, we propose an algorithm, referred to as OT data-driven filter (OT-DDF), that improves upon OTPF in two critical aspects: 
\begin{enumerate}
    \item We make the algorithm completely {\it data-driven}, by only requiring recorded data from the state and observations
without active usage of a simulator/oracle.
\item We make the online computations very light by replacing the online optimization with an offline training stage that finds a fixed map $T$ that conditions on a truncated measurement history $\{Y_t,Y_{t-1},\ldots,Y_{t-w+1}\}$ for some window size $w$. 
\end{enumerate}
These improvements become possible by making two critical assumptions about the model: (A1) the process $(X_t,Y_t)$ is stationary; (A2) the filter dynamics~\eqref{eq:filter-update} is stable. Precise statements of the assumptions appear in Sec.~\ref{sec:assumptions}.

The rest of the paper is organized as follows: 
Sec.~\ref{sec:problem_formulation} includes the mathematical setup and the modeling assumptions; Sec.~\ref{sec:OTDDF} contains the proposed methodology accompanied with an error analysis; and section~\ref{sec:numerics} presents several numerical experiments that serve as proof of concept and compares the proposed algorithm with alternative approaches. 

    
    \section{Problem Formulation}\label{sec:problem_formulation}
    \subsection{Mathematical setup}
\label{sec:filtering-problem}
In this paper, we consider a discrete-time stochastic dynamic system given by the update equations
\begin{subequations}\label{eq:model}
\begin{align}\label{eq:model-dyn}
    X_{t} &\sim a_t(\cdot|X_{t-1}) , \quad X_0 \sim \pi_0 \\\label{eq:model-obs}
    Y_t &\sim h_t(\cdot|X_t)
\end{align}
\end{subequations}
for $t=1,2,\ldots $ where $X_t\in \mathbb{R}^n$ is the hidden state of the system, $Y_t \in \mathbb{R}^m$ is the observation, $\pi_0$ is the probability distribution for the initial state $X_0$, $a_t( x'|x)$ is the transition kernel from $X_{t-1}=x$ to $X_t=x'$,  and $h_t(y|x)$ is the  likelihood of 
observing $Y_t=y $ given $X_t=x$. 

The dynamic and observation models are used to introduce the following propagation  and  conditioning operators
\begin{subequations}
    \begin{align}
    \text{(propagation)}~ \pi \mapsto \mathcal A_t (\pi) &:= \int_{\mathbb R^n} a_t(\cdot|x) \pi(x)d x,\\
  \text{(conditioning)}~ \pi \mapsto \mathcal B_t(y)(\pi) &:= \frac{h_t(y|\cdot)\pi(\cdot)}{\int_{\mathbb R^n} h_t(y|x) \pi(x)d x},
  \label{eq:Bayesian}   
\end{align}
\end{subequations} 
for an arbitrary probability distribution $\pi$. The propagation operator $\mathcal A_t$ represents the update for the distribution of the state according to the dynamic model~\eqref{eq:model-dyn}, and the 
conditioning operator $\mathcal B_t$ represents Bayes' rule that carries out the conditioning according to the observation model~\eqref{eq:model-obs}. The composition of these maps  is denoted by
\begin{equation*}
    F_t(y) := \mathcal B_t(y) \circ \mathcal A_t
\end{equation*}
and consecutive application of $F_t$ is denoted by
\begin{equation*}
    F_{t,s}(y_t,\ldots,y_{s+1}) \!:= F_t(y_t) \circ F_{t-1}(y_{t-1}) \circ \ldots \circ F_{s+1}(y_{s+1}),
\end{equation*}
for $t>s$. For simplicity, hereon, we introduce  the compact notation
    $y_{t,s}:=\{y_t,\ldots,y_{s+1}\}$ for $t>s\geq 0$. 

We are interested in two conditional distributions:
\begin{itemize}
    \item {\bf Exact posterior:} The exact posterior $\pi_t$ is defined as the conditional distribution of $X_t$ given $Y_{t,0}:=\{Y_t,\ldots,Y_1\}$, i.e. 
\begin{align}\label{eq:exact-posterior}
    \pi_t(\cdot):= \mathbb P(X_t\in \cdot  \mid Y_{t,0}).
\end{align}
In terms of our notation earlier, it can be identified via
\begin{align}\label{eq:exact-posterior-F}
    \pi_t = F_{t,0}(Y_{t,0})(\pi_0).
\end{align}
\item {\bf Truncated posterior:} The truncated posterior, denoted by $\pi_{t,s}^\mu$, is defined as the conditional distribution of $X_t$, given $Y_{t,s}:=\{Y_{s+1},\ldots,Y_t\}$, with the prior distribution $X_s \sim \mu$, i.e.
\begin{align}\label{eq:truncated-posterior}
    \pi_{t,s}^\mu(\cdot):= \mathbb P_{X_s\sim \mu}(X_t\in \cdot \mid Y_{t,s}).
\end{align}
It is given by the equation 
\begin{equation}\label{eq:truncated-posterior-F}
    \pi_{t,s}^\mu=F_{t,s}(Y_{t,s})(\mu). 
\end{equation}
\end{itemize}

\subsection{Modelling assumptions}
\label{sec:assumptions}

To ensure the applicability of our proposed method, we make the following two critical assumptions. 
 \begin{assumption}\label{assump:system_invariant}
The stochastic process $(X_t,Y_t)$ is stationary. In particular, the model~\eqref{eq:model} is time invariant, i.e. $a_t=a$ and $h_t=h$ for all $t=1,2,\ldots$, and $\pi_0$ is equal to the unique stationary distribution of $\mathcal A$, i.e. $\mathcal A \pi_0 = \pi_0$.  The stationary distribution has finite second moments, and it admits a density with respect to the Lebesgue measure.
\end{assumption}
\medskip 
This assumption implies that 
\begin{equation}
    F_{t,s} = F_{t-s,0},~\forall t>s\geq 0. 
\end{equation}

\begin{remark}
    The assumption that $\pi_0$ is equal to the stationary distribution is made to facilitate the error analysis in Sec.~\ref{sec:error_analysis}. This assumption may be replaced with geometric ergodicity of the Markov process $X_t$ at the cost of the extra term appearing in the error bound~\eqref{eq:pitP-bound}. In our numerical simulations, we simulate the true model for a {\it burn-in} period to ensure the probability distribution of $X_t$ is close to being stationary. 
 \end{remark}

The second assumption is related to the stability of the filtering dynamics~\eqref{eq:exact-posterior-F}. 
Consider the following metric on (possibly random) probability measures $\mu,\nu$:
\begin{equation}\label{eq:metric}
    d(\mu,\nu) := \sup_{g \in \mathcal G} \sqrt{\mathbb{E}\left|\int g d \mu - \int g d \nu\right|^2}
\end{equation}
where the expectation is over the possible randomness of the probability measures $\mu$ and $\nu$, and $\mathcal G:=\{g:\Re^n \to \Re;\,|g(x)|\leq 1, |g(x)-g(x')|\leq \|x-x'\|,\quad \forall x,x'\}$ is the space of functions that are uniformly bounded by one and uniformly Lipschitz with {a constant} smaller than one (this metric is also known as the dual bounded-Lipschitz distance~\cite{van2009uniform}). 

 \begin{assumption}\label{assump:stability}
 The filter 
  is uniformly geometrically stable, i.e.  
$\exists \lambda\in (0,1)$ and a positive constant $C>0$ such that for all $\mu,\nu$ and $ t>s\geq 0$ it holds that
\begin{equation}\label{eq:stability}
d(F_{t,s}(Y_{t,s})(\mu), F_{t,s}(Y_{t,s})(\nu)) \leq C\lambda^{t-s}d(\mu,\nu).
\end{equation}
\end{assumption}
\medskip

\begin{remark}
    The filter stability is a natural assumption that ensures the applicability of a numerical filtering algorithm. Similar stability assumptions also appear in~\cite{del2001stability,van2009uniform} for the error analysis of particle filters. However, finding necessary and sufficient conditions that ensure filter stability is challenging; see~\cite[Remark 1]{al2023optimal} and~\cite{van2009observability,crisan2011oxford,kim2022duality} for 
    example conditions that ensure filter stability. 
\end{remark}

\subsection{Objective}
In the usual filtering setup, the dynamic and observation models, $a$ and $h$, are known. However,   
we assume that these models are unknown. Instead, we have access to  $J$ recorded independent state-observation trajectories  of length $t_f$, i.e. $\{X^j_0,(X^j_1,Y^j_1),\ldots,(X^j_{t_f},Y^j_{t_f})\}_{j=1}^J$. Then, our objective is 
\begin{align*}
    \text{Given:}\quad & \left\{X^j_0,(X^j_1,Y^j_1),\ldots,(X^j_{t_f},Y^j_{t_f})\right\}_{j=1}^J\\
    \text{Approximate:}\quad   &\pi_t= \mathbb P(X_t\in \cdot \mid Y_t,\ldots,Y_1)\quad \forall t\geq 0, \\ & \text{for a new 
    set of observations } \{Y_t, \dots, Y_1\}.
\end{align*}

    \section{The Optimal transport Data-Driven Filter}\label{sec:OTDDF}
    \subsection{OT formulation for conditioning}\label{sec:OT-formulation}
The proposed methodology is based on the OT formulation of the Bayes' law that is used to represent conditional distributions as a push-forward of 
 OT maps~\cite{taghvaei2022optimal,hosseini2023conditional}. 
Consider a joint probability distribution $\nu_{XY}$ and its conditional   $\nu_{X|Y}$. Then, the goal is to find a map $\overline{T}$ such that 
\begin{align}\label{eq:OT-formulation-push-forward}
    \nu_{X|Y}(\cdot|y) = \overline T(\cdot,y)_{\#}\eta_X
\end{align}
where $\eta_X$ is an arbitrary probability distribution. Furthermore, the map $\overline T(\cdot,y)$ is the OT map from $\eta_X$ to $\nu_{X|Y}(\cdot|y)$ for $\nu_Y$-almost every $y$; where we used 
$\nu_Y$ to denote the $Y$-marginal of $\nu_{XY}$. The OT formulation is useful because the  map $\overline T$ can be obtained as the solution to a max-min stochastic optimization problem
\begin{align} \label{eq:OT-formulation-min-max}
    \max_{ f \in c\text{-Concave}_x}\,
    \min_{T \in \mathcal{M}(\eta)}\, J_{\eta,\nu}(f,T)
\end{align}
where $\eta= \eta_X \otimes \nu_Y$ is the independent coupling of $\eta_X$ and $\nu_Y$, $\mathcal{M}(\eta)$ denotes
the set of  $\eta$-measurable maps, $f \in c\text{-Concave}_x$ means $x\mapsto \frac{1}{2}\|x\|^2 - f(x,y)$ is convex in $x$ for all $y$~\cite[Def. 2.33]{villani2003topics}, and
the objective function
\begin{align*}
J_{\eta,\nu}(f,T)& := \mathbb{E}_{(X,Y) \sim \nu} [ f(X,Y)] +\\ 
    &\mathbb{E}_{(X,Y) \sim \eta} [\frac{1}{2}\|T(X,Y)- X\|^2 - f(T( X,Y),Y)]. 
\end{align*}
The well-posedness of the max-min problem is stated in the following theorem. 
\begin{theorem}\label{thm:conditioning}
    If $\eta_X$ has a finite second moment and it is absolutely continuous with respect to the Lebesgue measure, then the max-min problem~\eqref{eq:OT-formulation-min-max} admits a unique optimal pair $(\overline f,\overline T)$, modulo additive constant shifts for $\overline{f}$, and the relationship~\eqref{eq:OT-formulation-push-forward} holds $\nu_Y$-almost everywhere.  
\end{theorem}

\begin{remark}{The proof of this result appears in~\cite[Proposition 2.3]{al2023highdim}} which is an extension of the existing results \cite[Theorem 2.3]{carlier2016vector} and~\cite[Theorem 2.4]{kovachki2020conditional}. The extension to the Riemannian manifold and infinite-dimensional functional space settings appear in~\cite{grange2023computational} and~\cite{hosseini2023conditional}, respectively. 
\end{remark}

The following result relates the error in approximating the conditional distribution with the optimality gap of solving the max-min problem. In particular, let $(\hat f,\hat T)$ be the output of an algorithm that approximately solves~\eqref{eq:OT-formulation-min-max}
and consider $\hat T(\cdot,Y)_{\#}\eta_X$ 
as an approximation to $\nu_{X|Y}(\cdot|y)$. Define the max-min optimality gap
\begin{equation}\label{eq:opt-gaps}
\begin{aligned}
        \epsilon(\hat f,\hat T) := &J_{\eta,\nu}(\hat f,\hat T) - \min_{T} J_{\eta,\nu}(\hat f,T) \\+& \max_{f}\min_{T}J_{\eta,\nu}(f,T) - \min_{T} J_{\eta,\nu}(\hat f,T),  
\end{aligned}
\end{equation}
where the first term specifies the gap in the minimization, and the second term specifies the gap in the maximization.
Then we have the following lemma, the proof of which is given in
~\cite[Proposition 2.4]{al2023highdim}.
\begin{lemma}\label{lem:opt-gap}
    Consider the assumptions of Theorem~\ref{thm:conditioning}. Then, for any pair $(\hat f,\hat T)$ such that $x \mapsto \frac{1}{2}\|x\|^2-\hat f(x,y)$ is $\alpha$-strongly 
convex in $x$ for all $y$, we have the bound
\begin{equation}\label{eq:T-error-opt-gap}
    d(\hat T(\cdot,Y)_{\#}\eta_X,\nu_{X|Y}(\cdot|Y))\leq \sqrt{\frac{4}{\alpha}\epsilon(\hat f,\hat T)}.
\end{equation}
\end{lemma}

\subsection{Proposed methodology}
The proposed methodology is summarized in four steps: 

\newP{Step 1} We propose to approximate the truncated posterior~\eqref{eq:truncated-posterior} instead of the exact posterior~\eqref{eq:exact-posterior}. This step introduces an error that is bounded due to filter stability Assumption~\ref{assump:stability}.  In particular, replacing $\nu$ by $\pi_s$ in~\eqref{eq:stability}, and the fact that $F_{t,s}(Y_{t,s})(\pi_s)=\pi_t$ and $F_{t,s}(Y_{t,s})(\mu)=\pi^\mu_{t,s}$ , we conclude the bound
\begin{equation}\label{eq:truncation-bound}
d(\pi^{\mu}_{t,s}, \pi_t) \leq C\lambda^{t-s}d(\mu,\pi_s).
\end{equation}  
The error bound can be made arbitrarily small by choosing a large window size $w:=t-s$ and assuming a uniform bound on $d(\mu,\pi_s)$ for all $s$. 

\newP{Step 2} 
We use the OT formulation~\eqref{eq:OT-formulation-push-forward} with the min-max problem~\eqref{eq:OT-formulation-min-max} to characterize the truncated posterior. In order to do so, we choose the target distribution $\nu$ to be the joint distribution of $(X_t,Y_{t,s})$ where $X_t$ and $Y_{t,s}:=\{Y_t,\ldots,Y_{s+1}\}$ are generated using the stochastic model~\eqref{eq:model-dyn}-\eqref{eq:model-obs} with $X_s \sim \mu$. And we choose the source distribution $\eta$ to be equal to the independent coupling of $X_s\sim \mu$ (i.e. $\eta_X=\mu$) and $Y_{t,s}$, i.e. 
\begin{equation}\label{eq:target-source}
\begin{aligned}
    \nu &= \text{Law}(X_t,Y_{t,s}) &&\text{with}\quad X_s \sim \mu,\\
  \eta &= \text{Law}(X_s)\otimes \text{Law}(Y_{t,s})\quad &&\text{with}\quad X_s \sim \mu. 
\end{aligned}
\end{equation}
With this setup, the conditional distribution $\nu_{X|Y}$ equals the truncated posterior $\pi_{t,s}^\mu$. Let $\overline T^\mu_{t,s}$ denote the optimizer of the max-min problem~\eqref{eq:OT-formulation-min-max} with $\nu$ and $\eta$ chosen as explained above. Then, the relationship \eqref{eq:OT-formulation-push-forward} implies that
\begin{align*}
    \pi_{t,s}^\mu(\cdot) = \overline T_{t,s}^\mu(\cdot,Y_{t,s})_{\#} \mu. 
\end{align*}
Using the fact that $\pi_{t,s}^\mu$ is also given by~\eqref{eq:truncated-posterior-F}, we can also conclude the identity 
\begin{align}\label{eq:F-T-identity}
    F_{t,s}(Y_{t,s})(\mu) = \overline T_{t,s}^\mu(\cdot,Y_{t,s})_{\#} \mu,
\end{align}
for all probability distributions $\mu$. 

 \newP{Step 3} By  the time-invariance  Assumption~\ref{assump:system_invariant} and the identity~\eqref{eq:F-T-identity} we conclude, with $w= t-s$, that
\begin{align*}
    \overline T^\mu_{t,s} =      \overline  T^\mu_{w,0} \quad \forall\,t>s\geq 0. 
\end{align*}

\newP{Step 4} We use the recorded data  to numerically approximate the map $\overline T^\mu_{w,0}$ by solving  the max-min problem~\eqref{eq:OT-formulation-min-max}. In this problem, the target distribution $\nu$ is equal to the joint distribution of $(X_w,Y_{w,0})$ with $X_0 \sim \mu$ and the source distribution $\eta$ is equal to the independent coupling of $X_0$ and $Y_{w,0}$. The source and target distributions are empirically approximated as 
\begin{align}\label{eq:nu-hat}
    \nu &\approx  \hat \nu:=\frac{1}{J}\sum_{j=1}^J \delta_{(X^j_{w},Y^j_{w,0})},~
    \eta \approx \hat \eta := \frac{1}{J}\sum_{j=1}^J \delta_{(X^{\sigma_j}_{0},Y^j_{w,0})}
\end{align}
where $\{X^j_0,(X^j_1,Y^j_1),\ldots,(X^j_w,Y^j_w)\}_{j=1}^J$ are independent realizations of the state and observation available from recorded data, and $\{\sigma_1,\ldots,\sigma_J\}$ is a random permutation of $\{1,2,\ldots,N\}$. We use stochastic optimization 
methods to approximately solve the resulting optimization problem by searching for the functions $f$ and map $T$ inside the parameterized classes $\mathcal F$ and $\mathcal T$ respectively {(algorithm details appear in~Sec.~\ref{sec:numerical_alg})}. We denote the resulting approximate pair by $(\hat f^\mu_{w,0},\hat{T}^\mu_{w,0})$, i.e. 
\begin{equation}\label{eq:numerical_eq}
    (\hat f^\mu_{w,0},\hat{T}^\mu_{w,0}) \leftarrow \max_{f \in \mathcal F}\,\min_{T \in  \mathcal T}\,J_{\hat \eta,\hat \nu}(f,T). 
\end{equation}

\newP{Summary} The four-step procedure is summarized as 
\begin{align*}
    \pi_t \overset{(1)}{\approx} \pi_{t,s}^\mu &\overset{(2)}{=} \overline T^\mu_{t,s}(\cdot,Y_{t,s})_{\#} \mu \\&\overset{(3)}{=} \overline T^\mu_{w,0}(\cdot,Y_{t,s})_{\#}\mu \overset{(4)}{\approx} \hat T^\mu_{w,0}(\cdot,Y_{t,s})_{\#} \mu, 
\end{align*}
where the first step is approximation due to truncation, the second step is identity using the OT formulation, the third step is identity using the stationarity of the model, and the fourth step is numerical and algorithmic approximation.

    \subsection{Error analysis}\label{sec:error_analysis}
    The objective of the error analysis is to bound the error 
between the exact posterior $\pi_t$ and the approximation $\hat T^\mu_w(\cdot,Y_{t,s})_{\#}\mu$ obtained from the four step procedure. Two steps of the procedure involve approximation. The first step is due to the truncation, and the fourth step is due to approximation in solving the max-min problem.  The error, due to the first step, is bounded using the filter stability according to~\eqref{eq:truncation-bound}. The error, due to the second step, is bounded by the optimality gap using Lemma~\ref{lem:opt-gap}. The two results are combined to conclude an error bound under the following assumptions about the algorithm. 
\begin{assumption} \label{assump:algorithm} The following conditions, regarding the algorithm, hold: 
\begin{enumerate}
    \item[A.3a] 
   $\mu$  is equal to the stationary distribution of $X_t$.\footnote{In numerical experiments this can be approximately satisfied using a large enough burn-in time. } 
\item[A.3b] $\exists M>0$ such that $d(\pi_t,\mu)<M$ for all $t$. 
\item [A.3c] $x \mapsto \frac{1}{2}\|x\|^2-\hat f_{w,0}^\mu(x,y)$ is $\alpha$-strongly 
convex in $x$ for all $y$.
    \end{enumerate}
    
\end{assumption}

\begin{proposition}\label{prop:mean-field} Consider 
a window side $w>0$ and 
suppose Assumptions~\ref{assump:system_invariant},~\ref{assump:stability}, and~\ref{assump:algorithm} hold. 
    Then,    for all $t>w$,
      \begin{align}\label{eq:pitP-bound}
        d(\hat T^\mu_{w,0}(\cdot,Y_{t,t-w})_{\#}\mu,\pi_t) \leq
        C\lambda^w M +\sqrt{\frac{4}{\alpha}\Expect\epsilon(\hat f^\mu_{w,0},\hat T_{w,0}^\mu)},
    \end{align}
where the expectation is with respect to the randomness of training data and possibly the optimization procedure. 
\end{proposition}

\begin{remark}
The first term in the bound is due to the truncation and becomes small as the window size $w$ increases. The second term depends on the optimality gap, and it is expected to decrease as the class of functions $\mathcal{F,T}$ becomes more expressive and the number of samples $J$ becomes large. Analysis of the optimality gap is the subject of ongoing work, and {it is expected to follow from existing results for statistical generalization of optimal transport map estimation in~\cite{divol2022optimal} and approximation theory in~\cite{baptista2023approximation}.}
\end{remark}

\begin{proof}[Proof of Proposition~\ref{prop:mean-field}]
For simplicity, introduce the notation $\hat S := \hat T^\mu_{w,0}(\cdot,Y_{t,t-w})$ and $\overline S := \overline T^\mu_{w,0}(\cdot,Y_{t,t-w})$.  Upon the application of the triangle inequality and the identity $\pi_{t,t-w}^\mu  =\overline S_{\#} \mu$, we obtain the decomposition 
\begin{align*}
    d(&\hat S_{\#}\mu,\pi_t)  \leq   
     d(\hat S_{\#}\mu,\overline S_{\#}\mu) + d(\pi^\mu_{t,t-w},\pi_t). 
\end{align*}
Application of the bound~\eqref{eq:truncation-bound} and Assumption A3b on $d(\pi^\mu_{t,t-w},\pi_t)$ concludes the first term in~\eqref{eq:pitP-bound}.  Application of inequality~\eqref{eq:T-error-opt-gap} to $d(\hat S_{\#}\mu,\overline S_{\#}\mu)$ concludes the second term in~\eqref{eq:pitP-bound}.  Assumption A3a about the probability distribution $\mu$ is required to ensure that the probability distribution of the observation process  $\nu_Y$ that is used for the max-min optimization is equal to the probability distribution of $Y_{t,t-w}$ that comes from the true model~\eqref{eq:model}.  
\end{proof}


        \section{Numerics}\label{sec:numerics}
    \subsection{The numerical algorithm}\label{sec:numerical_alg}

The OT-DDF algorithm consists of two stages: (i) an offline stage which approximately solves~\eqref{eq:numerical_eq} to obtain the transport map $\hat T^\mu_{w,0}$; (ii) an online stage which uses the truncated history of observations $Y_{t,t-w}$ and the learned map $\hat T^\mu_{w,0}$ to approximate the posterior $\pi_t$ for each time $t$.

In the offline stage, we use the ADAM stochastic optimization algorithm to solve~\eqref{eq:numerical_eq}. The algorithm consists of inner and outer loops. The inner loop consists of   $k_{inner}$ iterations of ADAM to update the map $T$, while the outer loop consists of $k_{outer}$ iterations of ADAM to update  $f$. The functions $f$ and $T$ are parameterized with neural networks with architectures described separately for each example.  The samples $(X^j_{t_0},Y^j_{t_0+w,t_0})_{j=1}^J$ are used to form the approximation of $\hat \nu $ and $\hat \eta $ in~\eqref{eq:nu-hat}. A burn-in time $t_0$ is included to ensure that the training data is stationary and assumption A3.a is satisfied approximately.  The offline stage is summarized in Algorithm~\ref{alg:ot-ddf}.

The subsequent stage is the online stage, where the output map $\hat T^\mu_{w,0}$ of algorithm~\ref{alg:ot-ddf} is used to approximate the posterior $\pi_t$ using the observations $Y_t$ that are received online. 

In particular, the posterior $\pi_t$ is approximated with the 
empirical measure $\frac{1}{N}\sum_{i=1}^N \delta_{X^i_t}$ where 
\[{{X}}_{t}^{i} = \hat T^\mu_{w,0}(X_{t_0}^i,Y_{t,t-w}), \]
and $\{X^i_{t_0}\}_{i=1}^N$  
are $N$ random samples from $\{X^j_{t_0}\}_{j=1}^J$.

\begin{algorithm}[t]
\caption{Offline Training of OT-DDF} 
\begin{algorithmic}
\STATE \textbf{Input:} Recorded data $\{X^j_0,(X^j_1,Y^j_1),\ldots,(X^j_w,Y^j_w)\}_{j=1}^J$, \\ burn-in time $t_0$, window $w$, batch size $b_s$, \\ architecture, optimizer and learning rates for $f,T$,\\inner and outer loop iterations $k_{inner}$, $k_{outer}$.

\STATE \textbf{Initialize:} initialize neural net $f,T$ weights $\theta_f,\theta_T$.
\STATE  Create a random permutation $\{\sigma_i\}_{i=1}^J$ 
\FOR{$k=1$ to $k_{outer}$}
\STATE Select random  batch $(X_{t_0}^{\sigma_i},X_{t_0+w}^i,Y_{t_0+w,t_0}^i)_{i=1}^{b_s}$ 
\STATE Define $T^i = T(  X_{t_0}^{\sigma_i},Y_{t_0+w,t_0}^i)$  for $i=1,\ldots,b_s$
\FOR{$j=1$ to $k_{inner}$}
\STATE  Update $\theta_T$ to minimize \(\frac{1}{b_s}\sum_{i=1}^{b_s} \big[\frac{1}{2}\| X_{t_0}^{\sigma_i}- T^i\|^2 - f(T^i,Y_{t_0+w,t_0}^i)\big]\)
\ENDFOR
\STATE Update $\theta_f$ to minimize\\ \( \frac{1}{b_s}\sum_{i=1}^{b_s} \big[-f( X_{t_0+w}^i,Y_{t_0+w,t_0}^i) + f(T^i,Y_{t_0+w,t_0}^i)\big]\)
\ENDFOR

\STATE \textbf{Output:} Map $\hat T^\mu_{w,0}=T$.
\end{algorithmic}
\label{alg:ot-ddf}
\end{algorithm}

The proposed OT-DDF method
is evaluated against three other algorithms: the Ensemble Kalman filter (EnKF)~\cite{evensen2006}, OTPF~\cite{al2023optimal,al2023highdim}, and the
sequential importance resampling (SIR) PF~\cite{doucet09}. The details for all of the three algorithms appear in~\cite{al2023highdim}, and the numerical code used to produce the results is available online\footnote{\url{https://github.com/Mohd9485/OT-DDF}}.

\subsection{Linear dynamics with linear and quadratic observation models}
Consider 
\begin{subequations}\label{eq:model-example}
\begin{align}
    X_{t} &= \begin{bmatrix}
        \alpha & \sqrt{1-\alpha^2}
        \\
        -\sqrt{1-\alpha^2} & \alpha
    \end{bmatrix}
    X_{t-1} + \sigma V_t\\
    Y_t &= h(X_t) + \sigma W_t
\end{align}
\end{subequations}
for $t=1,2,\dots$ where $X_t\in \mathbb R^2$, $Y_t \in \mathbb R$, $\{V_t\}_{t=1}^\infty$ and $\{W_t\}_{t=1}^\infty$ are i.i.d sequences of $2$-dimensional and one-dimensional standard Gaussian random variables, $\alpha=0.9$ and $\sigma^2=0.1$. Two observation functions are of interest:
\begin{align*}
    h(X_t)=X_t(1), \quad \text{and}\quad  h(X_t)=X_t(1)^2
\end{align*}
where $X_t(1)$ is the first component of the vector $X_t$. We refer to these observation models as linear and quadratic, respectively.

We implemented algorithm~\ref{alg:ot-ddf} with different window sizes $w=1,10,50$. The burn-in time was $t_{0} = 100-w$. The functions $f$ and $T$ were parameterized as ResNets with one and two blocks of size $64$ and $48$, respectively. The optimization learning rates for $f$ and $T$ was $10^{-3}$ and $5\times10^{-4}$ with  $k_{inner} = 10$, $k_{outer} = 12000$, and batch size $b_s=64$.

\begin{figure*}[t]

         \centering
         \includegraphics[width=1\hsize,trim={105 10 20 50},clip]{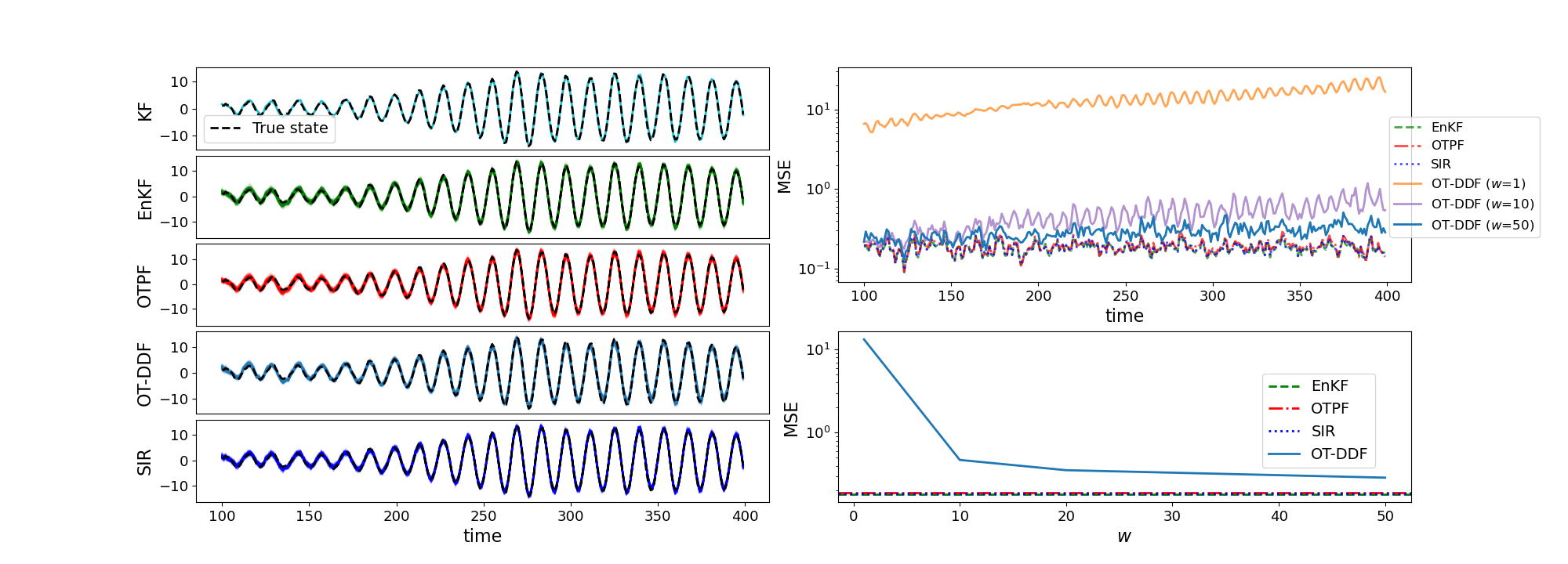}
     
     \caption{
     Numerical results for the linear dynamic example with linear observation function. The left column shows the trajectory of the second component of the particles along with the trajectory of the true state, where $w=50$ for the OT-DDF method. The right column shows the  MSE in estimating the state as a function of time in the upper corner and as a function of the window size $w$  in the lower corner.
     }
    \label{fig:X}
\end{figure*}

\begin{figure*}[t]

         \centering
         \includegraphics[width=1\hsize,trim={105 10 20 50},clip]{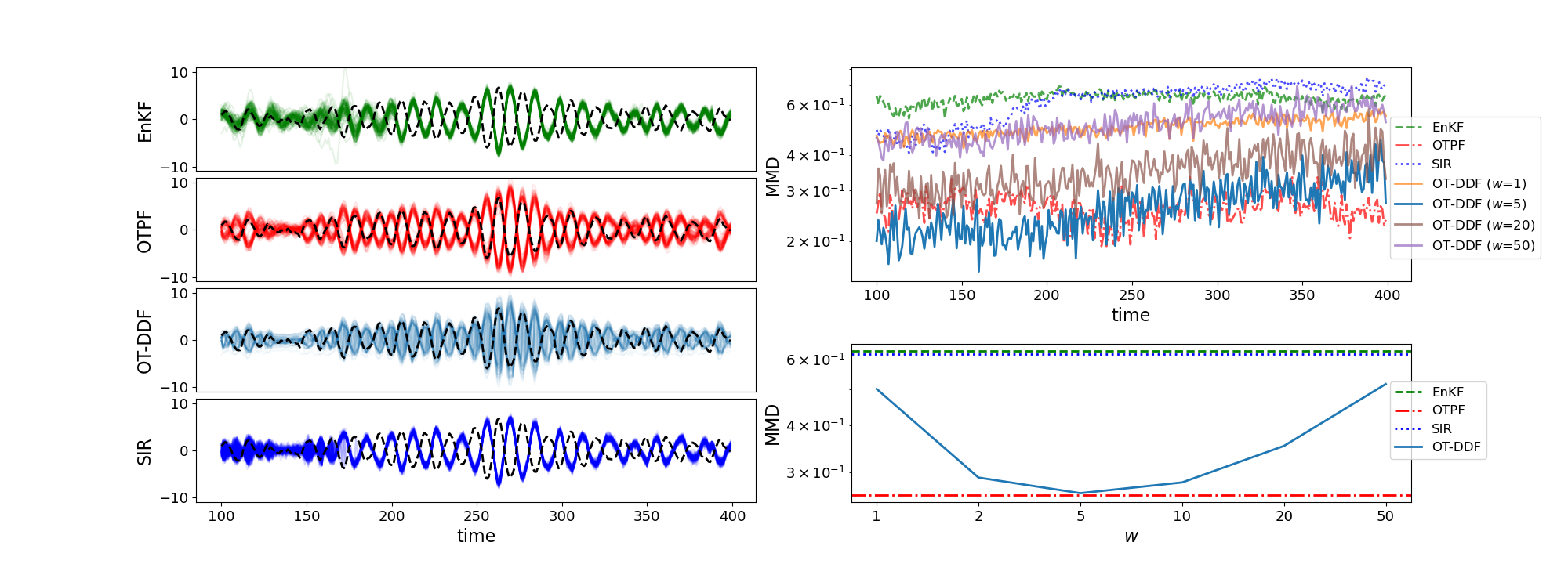}
     
     \caption{
         Numerical results for the linear dynamic example with quadratic observation function. The left column shows the trajectory of the second component of the particles along with the trajectory of the true state, where $w=5$ for the OT-DDF method.
     The right column shows the  MMD distance, with respect to the true posterior,  as a function of time in the upper corner and as a function of the window size $w$   in the lower corner.
     }
    \label{fig:XX}
\end{figure*}

The numerical results for the linear observation model are presented in Figure~\ref{fig:X}. The left column shows the trajectory of the particles along with the trajectory of the unobserved component $X(2)$ of the state for all methods and OT-DDF with $w=50$. We also included the Kalman filter (KF) because it provides the ground truth for this case.  The performance of the algorithms is quantified by computing the mean-squared-error (MSE) in estimating the true state $X_t$. The result is depicted in the top-right panel. The MSE is averaged over $50$ independent simulations.  The bottom right panel shows the time-averaged MSE for the OT-DDF method as the window size varies.  In the linear Gaussian setting, the KF is optimal and yields the least MSE.  The EnKF, OTPF, and SIR also provide the same performance as expected. The error for the OT-DDF is due to the window size and, as expected, decreases for larger windows.


Similar results for the quadratic observation model are presented in Figure~\ref{fig:XX}. In this case, the posterior is bimodal due to the symmetry in the model. The trajectories from the left panel show that OTPF and OT-DDF were able to capture the bimodal distribution while EnKF and SIR experienced mode collapse. To quantify the performance, we used the maximum mean discrepancy (MMD)~\cite{gretton2006kernel} with respect to the true posterior approximated with the SIR method with a large number of particles $(N=10^5)$. The result is presented in the top right panel, where the MMD is averaged over $10$ independent simulations. The bottom right panel shows the time-averaged MMD as the window size varies. The results show that the error initially decreases as the window size increases, but it starts to grow after a certain window size. We conjecture that this is due to the tradeoff between the two terms in the error bound~\eqref{eq:pitP-bound}. For small window size $w$, the truncation error is dominating, which decreases as $w$ increases. For a large window size, the optimization gap is dominating which grows as $w$ increases. This is due to the fact that the neural net architecture and the number of data points are kept fixed while the input size to the neural net increases.

\begin{figure*}[t]

         \centering
         \includegraphics[width=1\hsize,trim={105 10 20 50},clip]{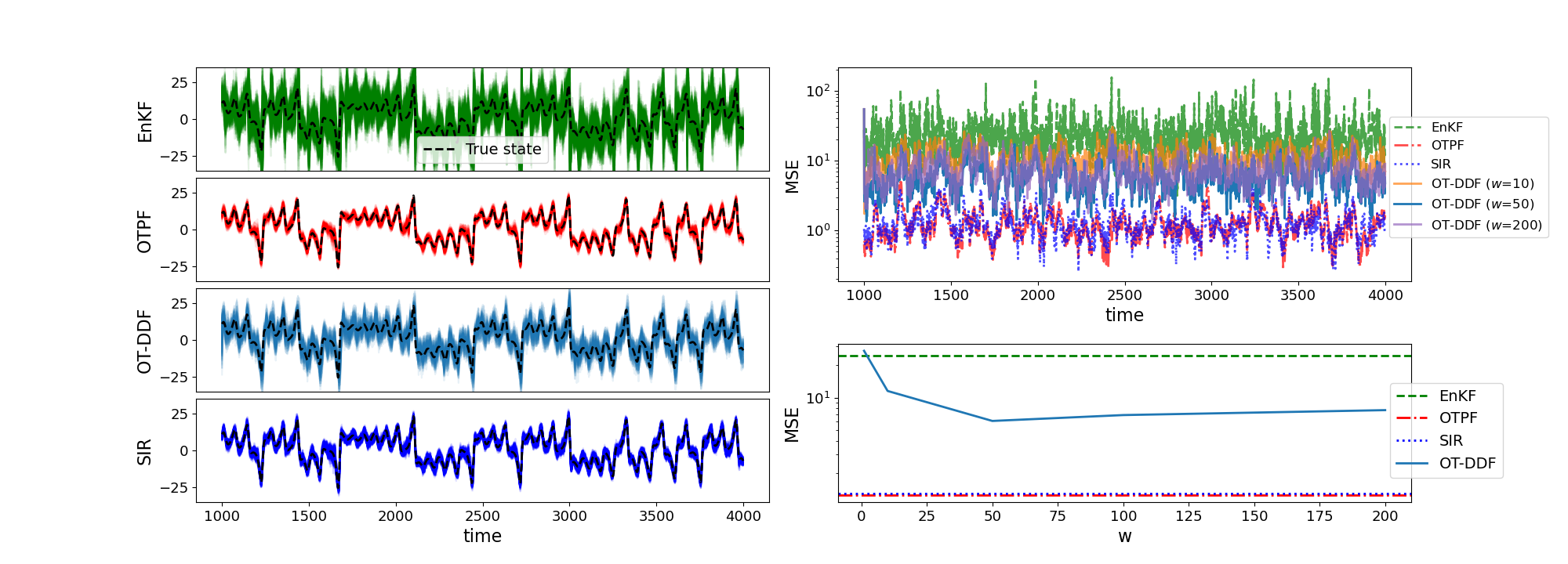}
     
     \caption{
     Numerical results for the Lorenz 63 example. The left column shows the trajectory of one of the unobserved components of the particles along with the trajectory of the true state, where $w=50$ for the OT-DDF method (the other components exhibit similar behavior). The right column shows the  MSE, in estimating the true state, as a function of time in the upper corner and as a function of the window size $w$ in the lower corner.
     }
    \label{fig:L63}
\end{figure*}

\subsection{Lorenz 63}
We repeated our experiments on a discrete-time (with time-discretization step-size $\Delta t = 0.01$) chaotic Lorenz 63 model~\cite{lorenz1963deterministic} with the observation function $h(x)=x(1)$ with additive zero-mean Gaussian noise of variance $10^{-1}$. Algorithm~\ref{alg:ot-ddf} was implemented for window sizes $w=10,50,200$ and burn-in time of $t_{0} = 1000-w$. The functions $f$ and $T$ wereparameterized with  ResNets of hidden size  $32$, with learning rate of $10^{-3}$ and $5\times10^{-3}$, number of iterations $k_{inner} = 10$ and $k_{outer} = 15000$. 

The numerical results are presented in Figure~\ref{fig:L63}. The left panel shows the trajectory of the particles and the true state (only the second component is shown). The right column shows the  MSE of estimating the true state, averaged over $10$ independent simulations, as a function of time in the top and as a function of the window size in the bottom. 
The performance of the OT-DDF filter is expected to improve with further fine-tuning, increasing the iteration number of training and the number of parameters in the neural net.

We also report the wall-clock complexity of all algorithms in Table~\ref{table:time_complexity}. The simulations are carried out on a MAC STUDIO M2 Max with a 12‑core CPU, 30‑core GPU, and 64GB unified memory. The offline training time for the OT-DDF with window size $w=50$ is $46.29$ seconds. The computational time per one-time step update for all methods appears in Table~\ref{table:time_complexity}. The time complexity of all methods except the OTPF algorithm is at the same level of magnitude, which allows the OT-DDF algorithm to be implemented in an online setting once we have access to the map $\hat T^\mu_{w,0}$. 
The online time complexity of OT-DDF may be further reduced with the application of neural net compression techniques and parallel processing of particles. 
\begin{table}[h]
    \centering
        \caption{The time complexity for one-time step.}
    \begin{tabular}{|c|c|c|c|c|}
    \hline
         Method &EnKF & SIR & OTPF & OT-DDF \\
         \hline
         time & $1.7 \times 10^{-4}$ & $2.0 \times 10^{-4}$ & $6.8 \times 10^{-2}$ & $1.5 \times 10^{-4}$\\
         \hline
    \end{tabular}
    \label{table:time_complexity}
\end{table}

    \section{Discussion}\label{sec:discussion}
    We introduced OT-DDF, a completely data-driven nonlinear filtering algorithm applicable to models that admit stationary processes and stable filters. The method provides significant improvement to the original OTPF method in terms of computational cost by limiting the costly training of an OT map to an offline stage using recorded data leading to very fast computations during online inference. Preliminary error analysis and numerical experiments show that in comparison to OTPF, the loss in accuracy is not significant when the window size is chosen properly, and the optimization problem is solved to a reasonable accuracy.  Future work will aim to refine the error analysis further and explore the algorithm's scalability and adaptability to a wider range of applications.

    \bibliographystyle{plain}
    \bibliography{TAC-OPT-FPF,references}

\end{document}